\documentclass[10pt]{amsart}
\usepackage{amssymb,MnSymbol,times}
\usepackage{amsthm,amsmath}
\usepackage{cite}

\title{A characterization of barycentrically preassociative functions}

\author{Jean-Luc Marichal}
\address{Mathematics Research Unit, FSTC, University of Luxembourg \\
6, rue Coudenhove-Kalergi, L-1359 Luxembourg, Luxembourg} \email{jean-luc.marichal[at]uni.lu}

\author{Bruno Teheux}
\address{Mathematics Research Unit, FSTC, University of Luxembourg \\
6, rue Coudenhove-Kalergi, L-1359 Luxembourg, Luxembourg} \email{bruno.teheux[at]uni.lu}

\date{September 22, 2015}

\begin{document}

\theoremstyle{plain}
\newtheorem{theorem}{Theorem}
\newtheorem{lemma}[theorem]{Lemma}
\newtheorem{proposition}[theorem]{Proposition}
\newtheorem{corollary}[theorem]{Corollary}
\newtheorem{fact}[theorem]{Fact}
\newtheorem*{main}{Main Theorem}

\theoremstyle{definition}
\newtheorem{definition}[theorem]{Definition}
\newtheorem{example}[theorem]{Example}

\theoremstyle{remark}
\newtheorem*{conjecture}{Conjecture}
\newtheorem{remark}{Remark}
\newtheorem{claim}{Claim}

\newcommand{\N}{\mathbb{N}}
\newcommand{\Q}{\mathbb{Q}}
\newcommand{\R}{\mathbb{R}}

\newcommand{\ran}{\mathrm{ran}}
\newcommand{\dom}{\mathrm{dom}}
\newcommand{\id}{\mathrm{id}}
\newcommand{\med}{\mathrm{med}}
\newcommand{\ofo}{\mathrm{ofo}}
\newcommand{\Ast}{\boldsymbol{\ast}}

\newcommand{\bfu}{\mathbf{u}}
\newcommand{\bfv}{\mathbf{v}}
\newcommand{\bfw}{\mathbf{w}}
\newcommand{\bfx}{\mathbf{x}}
\newcommand{\bfy}{\mathbf{y}}
\newcommand{\bfz}{\mathbf{z}}

\newcommand{\length}[1]{{\vert #1 \vert}}

\newcommand\restr[2]{{
  \left.\kern-\nulldelimiterspace 
  #1 
  \right|_{#2} 
  }}

\begin{abstract}
We provide a characterization of the variadic functions which are barycentrically preassociative as compositions of length-preserving associative string functions with one-to-one unary maps. We also discuss some consequences of this characterization.
\end{abstract}



\maketitle

\section{Introduction}

Let $X$ and $Y$ be arbitrary nonempty sets. Throughout this paper we regard tuples $\bfx$ in $X^n$ as $n$-strings over $X$. We let $X^*=\bigcup_{n \geqslant 0} X^n$ be the set of all strings over $X$, with the convention that $X^0=\{\varepsilon\}$ (i.e., $\varepsilon$ denotes the unique $0$-string on $X$). We denote the elements of $X^*$ by bold roman letters $\bfx$, $\bfy$, $\bfz$. If we want to stress that such an element is a letter of $X$, we use non-bold italic letters $x$, $y$, $z$, etc. The \emph{length} of a string $\bfx$ is denoted by $|\bfx|$. For instance, $|\varepsilon|=0$. We endow the set $X^*$ with the concatenation operation, for which $\varepsilon$ is the neutral element, i.e., $\varepsilon\bfx =\bfx\varepsilon =\bfx$. For instance, if $\bfx\in X^m$ and $y\in X$, then $\bfx y\in X^{m+1}$. Moreover, for every string $\bfx$ and every integer $n\geqslant 0$, the power $\bfx^n$ stands for the string obtained by concatenating $n$ copies of $\bfx$. In particular we have $\bfx^0=\varepsilon$.

As usual, a map $F\colon X^n\to Y$ is said to be an \emph{$n$-ary function} (an \emph{$n$-ary operation on $X$} if $Y=X$). Also, a map $F\colon X^*\to Y$ is said to be a \emph{variadic function} (a \emph{string function on $X$} if $Y=X^*$; see \cite{LehMarTeh}). For every variadic function $F\colon X^*\to Y$ and every integer $n\geqslant 0$, we denote by $F_n$ the \emph{$n$-ary part} $F|_{X^n}$ of $F$.

Recall that a variadic function $F\colon X^*\to Y$ is said to be \emph{preassociative} \cite{MarTeh,MarTeh2} if, for any $\bfx,\bfy,\bfy',\bfz\in X^*$, we have
$$
F(\bfy) = F(\bfy') \quad\Rightarrow\quad F(\bfx\bfy\bfz) = F(\bfx\bfy'\bfz).
$$
Also, a variadic function $F\colon X^*\to Y$ is said to be \emph{barycentrically preassociative} (or \emph{B-preassociative} for short) \cite{MarTeh3} if, for any $\bfx,\bfy,\bfy',\bfz\in X^*$, we have
$$
|\bfy| = |\bfy'| \quad\mbox{and}\quad F(\bfy) = F(\bfy') \quad\Rightarrow\quad F(\bfx\bfy\bfz) = F(\bfx\bfy'\bfz).
$$

Contrary to preassociativity, B-preassociativity recalls the associativity-like property of the barycenter (just regard $F(\bfx)$ as the barycenter of a set $\bfx$ of identical homogeneous balls in $X=\R^n$). In descriptive statistics and aggregation function theory, this condition says that the aggregated value of a series of numerical values remains unchanged when modifying a bundle of these values without changing their partial aggregation.

B-preassociativity has been recently utilized by the authors in the following characterization of the \emph{quasi-arithmetic pre-mean functions}, thus generalizing the well-known Kolmogoroff-Nagumo's characterization of the quasi-arithmetic mean functions.

\begin{theorem}[\cite{MarTeh3}]\label{thm:KolmExt}
Let $\mathbb{I}$ be a nontrivial real interval, possibly unbounded. A function $F\colon \mathbb{I}^*\to\R$ is B-preassociative and, for every $n\geqslant 1$, the function $F_n$ is symmetric, continuous, and strictly increasing in each argument if and only if there are continuous and strictly increasing functions $f\colon\mathbb{I}\to\R$ and $f_n\colon\R\to\R$ $(n\geqslant 1)$ such that
$$
F_n(\bfx) ~=~ f_n\bigg(\frac{1}{n}\sum_{i=1}^n f(x_i)\bigg),\qquad n\geqslant 1.
$$
\end{theorem}

\begin{remark}
If we add the condition that every $F_n$ is idempotent (i.e., $F_n(x^n)=x$ for every $x\in X$) in Theorem~\ref{thm:KolmExt}, then we necessarily have $f_n=f^{-1}$ for every $n\geqslant 1$, thus reducing this result to Kolmogoroff-Nagumo's characterization of the quasi-arithmetic mean functions \cite{Kol30,Nag30}. However, there are also many non-idempotent quasi-arithmetic pre-mean functions. Taking for instance $f_n(x)=nx$ and $f(x)=x$ over the reals $\mathbb{I}=\R$, we obtain the sum function. Taking $f_n(x)=\exp(nx)$ and $f(x)=\ln(x)$ over $\mathbb{I}=\left]0,\infty\right[$, we obtain the product function.
\end{remark}

In this paper we show that B-preassociative functions can be factorized as compositions of length-preserving associative string functions with one-to-one unary maps. We also show how this factorization result generalizes a characterization of a noteworthy subclass of B-preassociative functions given by the authors in \cite{MarTeh3}. Finally, we mention some interesting consequences of this new characterization.

The terminology used throughout this paper is the following. The domain, range, and kernel of any function $f$ are denoted by $\dom(f)$, $\ran(f)$, and $\ker(f)$, respectively. The identity function on any nonempty set is denoted by $\id$. For every $n\geqslant 1$, the diagonal section $\delta_F\colon X\to Y$ of a function $F\colon X^n\to Y$ is defined as $\delta_F(x)=F(x^n)$.

\begin{remark}
Although B-preassociativity was recently defined by the authors \cite{MarTeh3}, the basic idea behind this definition goes back to 1931 when de Finetti \cite{deF31} introduced an associativity-like property for mean functions. Indeed, according to de Finetti, for a real function $F\colon\bigcup_{n\geqslant 1}\R^n\to\R$ to be considered as a mean, it is natural that it be ``associative'' in the following sense: for any $u\in X$ and any $\bfx,\bfy,\bfz\in X^*$ such that  $|\bfx\bfz|\geqslant 1$ and $|\bfy|\geqslant 1$, we have $F(\bfx\bfy\bfz) = F(\bfx u^{|\bfy|}\bfz)$ whenever $F(\bfy) = F(u^{|\bfy|})$.
\end{remark}

\section{Main results}

As mentioned in the introduction, in this section we mainly show that B-preassociative functions can be factorized as compositions of length-preserving associative string functions with one-to-one unary maps. This result is stated in Theorem~\ref{thm:fa7sfds}.

Recall that a string function $F\colon X^*\to X^*$ is said to be \emph{associative} \cite{LehMarTeh} if it satisfies the equation $F(\bfx\bfy\bfz)=F(\bfx F(\bfy)\bfz)$ for any $\bfx,\bfy,\bfz\in X^*$.

\begin{definition}
We say that a string function $F\colon X^*\to X^*$ is \emph{length-preserving} if $|F(\bfx)|=|\bfx|$ for every $\bfx\in X^*$, or equivalently, if $\ran(F_n)\subseteq X^n$ for every $n\geqslant 0$.
\end{definition}



Clearly, the identity function on $X^*$ is associative and length-preserving. The following example gives nontrivial instances of associative and length-preserving string functions. Further examples of associative string functions can be found in \cite{LehMarTeh}.


\begin{example}\label{ex:68vffd}
Let $(h_n)_{n\geqslant 1}$ be a sequence of unary operations on $X$. One can easily see that the length-preserving function $F\colon X^*\to X^*$ defined by $F_0(\varepsilon)=\varepsilon$ and
$$
F_n(x_1\cdots x_n) ~=~ h_1(x_1)\cdots h_n(x_n),\qquad n\geqslant 1,
$$
is associative if and only if $h_n\circ h_m=h_n$ for all $n,m\geqslant 1$ such that $m\leqslant n$. Using an elementary induction, one can also show that the latter condition is equivalent to $h_n\circ h_n = h_n$ and $h_{n+1}\circ h_n = h_{n+1}$ for every $n\geqslant 1$. To give an example, take any constant sequence $h_n=h$ such that $h\circ h=h$ (for instance, the positive part function $h(x)=x^+$ over $X=\R$). As a second example, consider the sequence $h_n$ of unary operations on $X=\{1,2,3,\ldots\}$ defined by $h_n(k)=1$ if $k\leqslant n+1$, and $h_n(k)=k$, otherwise.
\end{example}

\begin{proposition}\label{prop:s7f65}
Let $F\colon X^*\to X^*$ be a length-preserving function. Then $F$ is associative if and only if it is B-preassociative and satisfies $F_n=F_n\circ F_n$ for every $n\geqslant 0$.
\end{proposition}

\begin{proof}
To see that the necessity holds, we recall from \cite{LehMarTeh} that any associative string function is preassociative and hence B-preassociative. The second part of the statement is immediate. For the sufficiency, we merely observe that we have $F(F(\bfy))=F(\bfy)$ for every $\bfy\in X^*$ and therefore, by B-preassociativity, we also have $F(\bfx F(\bfy)\bfz)=F(\bfx\bfy\bfz)$ for every $\bfx\bfy\bfz\in X^*$, that is, $F$ is associative.
\end{proof}

The following proposition, established in \cite{MarTeh3}, shows how we can construct new B-preassociative functions from given B-preassociative functions.

\begin{proposition}[{\cite{MarTeh3}}]\label{prop:leftcomp56}
Let $F\colon X^*\to Y$ be a B-preassociative function and let $(g_n)_{n\geqslant 1}$ be a sequence of functions from $Y$ to a nonempty set $Y'$. If $g_n|_{\ran(F_n)}$ is one-to-one for every $n\geqslant 1$, then any function $H\colon X^*\to Y'$ such that $H_n=g_n\circ F_n$ for every $n\geqslant 1$ is B-preassociative.
\end{proposition}

Recall that a function $g$ is a \emph{quasi-inverse} \cite[Sect.~2.1]{SchSkl83} of a function $f$ if
$$
f\circ \restr{g}{\ran(f)}=\restr{\id}{\ran(f)}\qquad\mbox{and}\qquad\ran(\restr{g}{\ran(f)})=\ran(g).
$$
We denote the set of quasi-inverses of a function $f$ by $Q(f)$. Under the assumption of the Axiom of Choice (AC), the set $Q(f)$ is nonempty for any function $f$. In fact, the Axiom of Choice is just another form of the statement ``every function has a quasi-inverse''. Note also that the relation of being quasi-inverse is symmetric: if $g \in Q(f)$ then $f \in Q(g)$; moreover, we have $\ran(g)\subseteq\dom(f)$ and $\ran(f)\subseteq\dom(g)$ and the functions $\restr{f}{\ran(g)}$ and $\restr{g}{\ran(f)}$ are one-to-one.

\begin{lemma}\label{lemma:f67s5fds}
Assume AC and let $F\colon X^n\to Y$ be a function. For any $g\in Q(F)$, define the function $H\colon X^n\to X^n$ by $H=g\circ F$. Then we have $F=F\circ H$ and $H=H\circ H$. Moreover, the map $F|_{\ran(H)}$ is one-to-one.
\end{lemma}

\begin{proof}
By definition of $H$ we have $F\circ H=F\circ g\circ F=F$ and $H\circ H=g\circ F\circ g\circ F=g\circ F=H$. Also, the map $F|_{\ran(H)}=F|_{\ran(g)}$ is one-to-one.
\end{proof}

\begin{lemma}\label{lemma:asdf7fsad}
Assume AC and let $F\colon X^*\to Y$ be a function. The following assertions are equivalent.
\begin{enumerate}
\item[(i)] $F$ is B-preassociative.

\item[(ii)] For every sequence $(g_n\in Q(F_n))_{n\geqslant 1}$, the function $H\colon X^*\to X^*$ defined by $H_0(\varepsilon)=\varepsilon$ and $H_n=g_n\circ F_n$ for every $n\geqslant 1$ is associative and length-preserving.

\item[(iii)] There exists a sequence $(g_n\in Q(F_n))_{n\geqslant 1}$ such that the function $H\colon X^*\to X^*$ defined by $H_0(\varepsilon)=\varepsilon$ and $H_n=g_n\circ F_n$ for every $n\geqslant 1$ is associative and length-preserving.
\end{enumerate}
\end{lemma}

\begin{proof}
(i) $\Rightarrow$ (ii). Let $H\colon X^*\to X^*$ be defined as indicated in the statement. We know by Lemma~\ref{lemma:f67s5fds} that $H\circ H=H$ and $H$ is length-preserving. Since $g_n|_{\ran(F_n)}$ is one-to-one, we have that $H$ is B-preassociative by Proposition~\ref{prop:leftcomp56}. It follows from Proposition~\ref{prop:s7f65} that $H$ is associative.

(ii) $\Rightarrow$ (iii). Trivial.

(iii) $\Rightarrow$ (i). By Proposition~\ref{prop:s7f65}, $H$ is B-preassociative. For every $n\geqslant 1$, since $g_n|_{\ran(F_n)}$ is a one-to-one map from $\ran(F_n)$ onto $\ran(g_n)=\ran(H_n)$, we have $F_n=(g_n|_{\ran(F_n)})^{-1}\circ H_n$. By Proposition~\ref{prop:leftcomp56} it follows that $F$ is B-preassociative.
\end{proof}

We are now ready to present our main result, which gives a characterization of any B-preassociative function as a composition of a length-preserving associative string function with one-to-one unary maps.

\begin{theorem}\label{thm:fa7sfds}
Assume AC and let $F\colon X^*\to Y$ be a function. The following assertions are equivalent.
\begin{enumerate}
\item[(i)] $F$ is B-preassociative.

\item[(ii)] There exist an associative and length-preserving function $H\colon X^*\to X^*$ and a sequence $(f_n)_{n\geqslant 1}$ of one-to-one functions $f_n\colon\ran(H_n)\to Y$ such that $F_n=f_n\circ H_n$ for every $n\geqslant 1$.
\end{enumerate}
If condition (ii) holds, then for every $n\geqslant 1$ we have $f_n=F|_{\ran(H_n)}=F_n|_{\ran(H_n)}$, $f_n^{-1}\in Q(F_n)$, and we may choose $H_n=g_n\circ F_n$ for any $g_n\in Q(F_n)$.
\end{theorem}

\begin{proof}
(i) $\Rightarrow$ (ii). Let $H\colon X^*\to X^*$ be defined by $H_0(\varepsilon)=\varepsilon$ and $H_n=g_n\circ F_n$ for every $n\geqslant 1$, where $g_n\in Q(F_n)$. By Lemma~\ref{lemma:f67s5fds} we have $F_n=f_n\circ H_n$ for every $n\geqslant 1$, where $f_n=F_n|_{\ran(H_n)}$ is one-to-one. By Lemma~\ref{lemma:asdf7fsad}, $H$ is associative and length-preserving.

(ii) $\Rightarrow$ (i). $H$ is B-preassociative by Proposition~\ref{prop:s7f65}. By Proposition~\ref{prop:leftcomp56} it follows that also $F$ is B-preassociative.

If condition (ii) holds, then for every $n\geqslant 1$ we have $F_n\circ H_n=f_n\circ H_n\circ H_n=f_n\circ H_n$ and hence $F_n|_{\ran(H_n)}=f_n$. Moreover, since $f_n$ is one-to-one, we have $H_n=f_n^{-1}\circ F_n$ and hence $F_n\circ f_n^{-1}\circ F_n=F_n\circ H_n=f_n\circ H_n\circ H_n=f_n\circ H_n=F_n$, which shows that $f_n^{-1}\in Q(F_n)$.
\end{proof}

\begin{remark}
\begin{enumerate}
\item[(a)] It is clear that the trivial factorization $F_n=F_n\circ H_n$, where $H_n=\id$, holds for any function $F\colon X^*\to Y$. This observation could make us wrongly think that Theorem~\ref{thm:fa7sfds} is of no use. However, in our factorization $F_n=f_n\circ H_n$ the outer function $f_n$ has the important feature that it is one-to-one.

\item[(b)] Similarly to Theorem~\ref{thm:fa7sfds}, one can show \cite{LehMarTeh} that any preassociative function $F\colon X^*\to Y$ can be factorized as a composition $F=f\circ H$, where $H\colon X^*\to X^*$ is associative and $f\colon\ran(H)\to Y$ is one-to-one.
\end{enumerate}
\end{remark}

In the rest of this section we show how Theorem~\ref{thm:fa7sfds} can be particularized to some nested subclasses of B-preassociative functions, including the subclass of B-preassociative functions $F\colon X^*\to Y$ for which the equality $\ran(F_n)=\ran(\delta_{F_n})$ holds for every $n\geqslant 1$ (see \cite{MarTeh3}).

For any integers $m,n\geqslant 1$, define $X_m^0=X^0$ and
$$
X^n_m ~=~ \{\bfy z^{n-\min\{n,m\}+1}:\bfy z\in X^{\min\{n,m\}}\}.
$$
For instance $X^3_1=\{z^3: z\in X\}$, $X^3_2 = \{y z^2:y z\in X^2\}$, and $X^3_m = X^3$ for every $m\geqslant 3$.

Thus, we have $X^n_m=X^n$ if $m\geqslant n$ and $X^n_m = \{\bfy z^{n-m+1}:\bfy z\in X^m\}$ if $m\leqslant n$. It follows that for every $m\geqslant 1$ we have $X^n_m\subseteq X^n_{m+1}\subseteq X^n$.

\begin{definition}
Let $m\geqslant 1$ and $n\geqslant 0$ be integers. We say that a function $H\colon X^n\to X^n$ has an \emph{$m$-generated range} if $\ran(H)\subseteq X^n_m$. We say that a function $H\colon X^*\to X^*$ has an \emph{$m$-generated range} if $H_n$ has an $m$-generated range for every $n\geqslant 0$.
\end{definition}

\begin{fact}
If a function $H\colon X^n\to X^n$ has an $m$-generated range, then it has an $(m+1)$-generated range. If a function $H\colon X^*\to X^*$ has an $m$-generated range, then it is length-preserving.
\end{fact}

Let $m\geqslant 1$ and $n\geqslant 0$ be integers. The \emph{$m$-diagonal section} of a function $F\colon X^n\to Y$ is the map $\delta_{F}^m\colon X^{\min\{n,m\}}\to Y$ defined by $\delta_F^m=F$, if $n=0$, and $\delta_F^m(\bfy z)=F(\bfy z^{n-\min\{n,m\}+1})$ for every $\bfy z\in X^{\min\{n,m\}}$, otherwise. We clearly have $\ran(\delta_F^m)\subseteq \ran(\delta_F^{m+1})$.

\begin{definition}
Let $m\geqslant 1$ and $n\geqslant 0$ be integers. We say that a function $F\colon X^n\to Y$ is \emph{$m$-quasi-range-idempotent} if $\ran(F)=\ran(\delta_F^m)$.
\end{definition}

By definition, any $m$-quasi-range-idempotent function $F\colon X^n\to Y$ is $(m+1)$-quasi-range-idempotent. We also observe that the property of being $m$-quasi-range-idempotent is preserved under left composition with unary maps: if $F\colon X^n\to Y$ is $m$-quasi-range-idempotent, then so is $g\circ F$ for any map $g\colon Y\to Y'$, where $Y'$ is a nonempty set.

\begin{proposition}
If $F\colon X^*\to X^*$ is associative and $F_k$ has an $m$-generated range for some $k,m\geqslant 1$, then for any integer $p\geqslant 0$ the function $F_{k+p}$ is $(m+p)$-quasi-range-idempotent. In particular, $F_k$ is $m$-quasi-range-idempotent.
\end{proposition}

\begin{proof}
Let $\bfx\in X^p$ and $\bfx'\in X^k$. Then, there exists $\bfy z\in X^{\min\{k,m\}}$ such that
\begin{eqnarray*}
F_{k+p}(\bfx\bfx') &=& F_{k+p}(\bfx F_k(\bfx')) ~=~ F_{k+p}(\bfx\bfy z^{k-\min\{k,m\}+1})\\
 &=& F_{k+p}(\bfx\bfy z^{(k+p)-\min\{k+p,m+p\}+1}) ~=~ \delta_{F_{k+p}}^{m+p}(\bfx\bfy z),
\end{eqnarray*}
which shows that $\ran(F_{k+p})\subseteq\ran(\delta_{F_{k+p}}^{m+p})$. The converse inclusion is obvious.
\end{proof}

\begin{lemma}\label{lemma:as6f5sa}
Let $m,n\geqslant 1$ be integers. Any map $F\colon X^n\to Y$ satisfying $F=F\circ H$, where $H\colon X^n\to X^n$ has an $m$-generated range, is $m$-quasi-range-idempotent.
\end{lemma}

\begin{proof}
Since $\ran(H)\subseteq X^n_m$, we have $\ran(F)=\ran(F\circ H)\subseteq\ran(\delta_F^m)$. Since the converse inclusion $\ran(F)\supseteq\ran(\delta_F^m)$ holds for any map $F\colon X^n\to Y$, we have that $F$ is $m$-quasi-range-idempotent.
\end{proof}

\begin{lemma}
Under the assumptions of Lemma~\ref{lemma:f67s5fds}, if $F$ is $m$-quasi-range-idempotent for some $m\geqslant 1$, then $g$ can always be chosen so that $\ran(g)\subseteq X^n_m$ and therefore $H$ has an $m$-generated range. Conversely, if $H$ has an $m$-generated range for some $m\geqslant 1$, then $F$ is $m$-quasi-range-idempotent.
\end{lemma}

\begin{proof}
If $F$ is $m$-quasi-range-idempotent for some $m\geqslant 1$, then there always exists $g\in Q(F)$ such that $\ran(g)\subseteq X^n_m$; indeed, if $y\in\ran(F)=\ran(\delta_F^m)$, then we can take $g(y)\in (\delta_F^m)^{-1}\{y\}\subseteq X^n_m$. Therefore $H=g\circ F$ has an $m$-generated range. Conversely, if $H$ has an $m$-generated range for some $m\geqslant 1$, then $F$ is $m$-quasi-range-idempotent by Lemma~\ref{lemma:as6f5sa}.
\end{proof}

\begin{corollary}\label{cor:fds68}
For any $m\geqslant 1$, the equivalence in Lemma~\ref{lemma:asdf7fsad} holds if we add the condition that every $F_n$ $(n\geqslant 1)$ is $m$-quasi-range-idempotent in assertion (i) and the conditions that $\ran(g_n)\subseteq X^n_m$ $(n\geqslant 1)$ and $H$ has an $m$-generated range in assertions (ii) and (iii).
\end{corollary}

\begin{theorem}\label{thm:sfa76r5}
For any $m\geqslant 1$, the equivalence between (i) and (ii) in Theorem~\ref{thm:fa7sfds} still holds if we add the condition that every $F_n$ $(n\geqslant 1)$ is $m$-quasi-range-idempotent in assertion (i) and the condition that $H$ has an $m$-generated range in assertion (ii). In this case the condition $\ran(g_n)\subseteq X^n_m$ $(n\geqslant 1)$ must be added in the last part of the statement.
\end{theorem}

\begin{proof}
Follows from the results above.
\end{proof}

Setting $m=1$ in Theorem~\ref{thm:sfa76r5}, we immediately derive a factorization of any B-preasso{\-}ciative function whose $n$-ary part $F_n$ is $1$-quasi-range-idempotent for every $n\geqslant 1$. An alternative factorization for such functions is given in the following theorem, established in \cite{MarTeh3}. Recall that a function $F\colon X^*\to X\cup\{\varepsilon\}$ is \emph{barycentrically associative} (or \emph{B-associative} for short) \cite{MarTeh3} if it satisfies the equation $F(\bfx\bfy\bfz)=F(\bfx F(\bfy)^{|\bfy|}\bfz)$ for any $\bfx,\bfy,\bfz\in X^*$. (B-associativity is also known as \emph{decomposability}, see \cite{FodRou94,GraMarMesPap09}).

\begin{theorem}[{\cite{MarTeh3}}]\label{thm:FactoriAWRI-BPA237111}
Assume AC and let $F\colon X^*\to Y$ be a function. The following assertions are equivalent.
\begin{enumerate}
\item[(i)] $F$ is B-preassociative and $F_n$ is $1$-quasi-range-idempotent for every $n\geqslant 1$.

\item[(ii)] There exists a B-associative function $H\colon X^*\to X\cup\{\varepsilon\}$ such that $H(\varepsilon)=\varepsilon$ and a sequence $(f_n)_{n\geqslant 1}$ of one-to-one functions $f_n\colon\ran(H_n)\to Y$ such that $F_n=f_n\circ H_n$ for every $n\geqslant 1$.
\end{enumerate}
If condition (ii) holds, then for every $n\geqslant 1$ we have $F_n=\delta_{F_n}\circ H_n$, $f_n=\delta_{F_n}|_{\ran(H_n)}$, $f_n^{-1}\in Q(\delta_{F_n})$, and we may choose $H_n=g_n\circ F_n$ for any $g_n\in Q(\delta_{F_n})$.
\end{theorem}

We now show how Theorem~\ref{thm:FactoriAWRI-BPA237111} can be easily derived from Theorem~\ref{thm:sfa76r5}.

For every $m\geqslant 1$ and every $\bfx\in X^*$, denote by $\bfx_{[m]}$ the \emph{$m$-prefix} of $\bfx$, that is the string in $\bigcup_{i=0}^m X^i$ defined as follows: if $|\bfx|\leqslant m$, then $\bfx_{[m]}=\bfx$; otherwise, if $\bfx=\bfx'\bfx''$, with $|\bfx'|=m$, then $\bfx_{[m]}=\bfx'$.

If $H\colon X^*\to X^*$ has an $m$-generated range, then by definition it can be assimilated with the function $H_{[m]}\colon X^*\to \bigcup_{i=0}^m X^i$ defined by $H_{[m]}(\bfx)=H(\bfx)_{[m]}$. Indeed, $H$ can be reconstructed from $H_{[m]}$ by setting
$$
H(\bfx) ~=~
\begin{cases}
H_{[m]}(\bfx), & \mbox{if $|\bfx|\leqslant m$},\\
H_{[m]}(\bfx)z^{n-m}, & \mbox{otherwise},
\end{cases}
$$
where $z$ is the last letter of $H_{[m]}(\bfx)$.

Thus we can prove Theorem~\ref{thm:FactoriAWRI-BPA237111} from Theorem~\ref{thm:sfa76r5} as follows.

\begin{proof}[Proof of Theorem~\ref{thm:FactoriAWRI-BPA237111} as a corollary of Theorem~\ref{thm:sfa76r5}]
By setting $m=1$ in Theorem~\ref{thm:sfa76r5}, we see that $H$ has a $1$-generated range. By the observation above, $H$ can then be assimilated with $H_{[1]}$ through the identity $H(\bfx)=H_{[1]}(\bfx)^{|\bfx|}$ for every $\bfx\in X^*$. It is then clear that $H$ is associative if and only if $H_{[1]}$ is B-associative. The other parts of Theorem~\ref{thm:FactoriAWRI-BPA237111} follow immediately.
\end{proof}

\begin{remark}
The question of generalizing Theorem~\ref{thm:FactoriAWRI-BPA237111} by dropping the $1$-quasi-range-idem{\-}potent condition on every $F_n$ was raised in \cite{MarTeh3}. Clearly, Theorem~\ref{thm:fa7sfds} answers this question.
\end{remark}

\section{Some consequences of the factorization result}

Since any associative function $F\colon X^*\to X^*$ is preassociative and, in turn, B-preassocia{\-}tive, it can be factorized as indicated in Theorem~\ref{thm:fa7sfds}. Therefore, up to one-to-one unary maps, the associative string functions can be completely described in terms of length-preserving associative string functions, and similarly for the preassociative and B-preasso{\-}ciative functions. This is an important observation which shows that in a sense any of these nested classes can be described in terms of the smallest one, namely the subclass of associative and length-preserving string functions (see Figure~\ref{fig:BA}).

\begin{figure}[htbp]
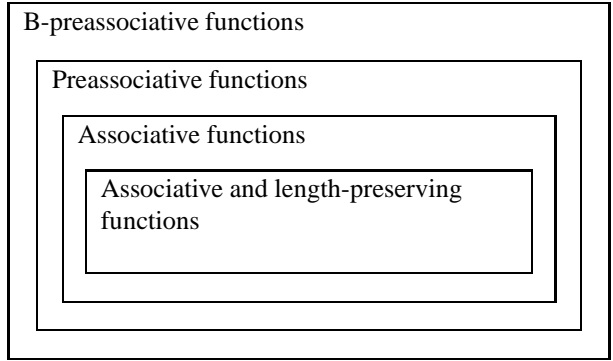
\centering
$$
\fbox{
\begin{minipage}[t]{0.6\textwidth}
B-preassociative functions
\vspace{1ex}
$$
\fbox{
\begin{minipage}[t]{0.9\textwidth}
Preassociative functions
\vspace{1ex}
$$
\fbox{
\begin{minipage}[t]{0.9\textwidth}
Associative functions
\vspace{1ex}
$$
\fbox{
\begin{minipage}[t]{0.9\textwidth}
Associative and length-preserving\\ functions
\vspace{3ex}
\end{minipage}
}
$$
\vspace{0.5ex}
\end{minipage}
}
$$
\vspace{0.5ex}
\end{minipage}
}
$$
\vspace{0.5ex}
\end{minipage}
}
$$
\caption{Nested subclasses of B-preassociative functions} \label{fig:BA}
\end{figure}

\begin{example}
Let $a\in X$ be fixed. Let the map $F\colon X^* \to X^*$ be defined inductively by $F(z) = z$ if $z \neq a$, $F(a) = \varepsilon$, and $F(\bfx z) = F(\bfx)F(z)$ for every $\bfx z\in X^*$. Thus defined, $F(\bfx)$ is obtained from $\bfx$ by removing all the `a' letters (if any). Since $F$ is associative (see \cite{LehMarTeh} for more details), it is B-preassociative and therefore it can be factorized as indicated in Theorem~\ref{thm:fa7sfds}. For every $n\geqslant 1$, define the function $g_n\colon\bigcup_{i=0}^n (X\setminus\{a\})^i\to X^n$ by $g_n(\bfx)=\bfx a^{n-|\bfx|}$. Since $F_n\circ g_n\circ F_n=F_n$ for every $n\geqslant 1$, we see that $g_n\in Q(F_n)$. By Theorem~\ref{thm:fa7sfds}, the function $H\colon X^*\to X^*$, defined by $H_0(\varepsilon)=\varepsilon$ and $H_n=g_n\circ F_n$ for every $n\geqslant 1$, is associative and length-preserving. Moreover, we have $F_n=f_n\circ H_n$ for every $n\geqslant 1$, where $f_n=F_n|_{\ran(H_n)}$. Thus defined, $H_n(\bfx)$ is obtained from $\bfx$ by moving all the `a' letters (if any) to the rightmost positions. For instance, $H_{11}(mathematics)=mthemticsaa$.
\end{example}

As observed in the previous section, setting $m=1$ in Theorem~\ref{thm:sfa76r5}, we can derive a factorization of any B-preassociative function whose $n$-ary part $F_n$ is $1$-quasi-range-idempotent for every $n\geqslant 1$ (Theorem~\ref{thm:FactoriAWRI-BPA237111}). In the following example, we derive a similar factorization explicitly directly from Theorem~\ref{thm:fa7sfds} (without using Theorem~\ref{thm:sfa76r5}).

\begin{example}\label{ex:5dfdas}
If we assume that $F_n$ is $1$-quasi-range-idempotent for every $n\geqslant 1$ in assertion (i) of Theorem~\ref{thm:fa7sfds}, then the factorization given in assertion (ii) can be obtained by defining $H_n=g_n\circ F_n$, where $g_n(x)=h_n(x)^n$ and $h_n\in Q(\delta_{F_n})$. Indeed, since $F_n$ is $1$-quasi-range-idempotent, we have
$$
(F_n\circ g_n\circ F_n)(\bfx) ~=~ (\delta_{F_n}\circ h_n\circ F_n)(\bfx) ~=~ F_n(\bfx),
$$
which shows that $g_n\in Q(F_n)$.
\end{example}

%

It is clear that the B-associativity property, originally defined for functions $F\colon X^*\to X\cup\{\varepsilon\}$ can be immediately extended to string functions $F\colon X^*\to X^*$.

\begin{definition}
We say that a string function $F\colon X^*\to X^*$ is \emph{barycentrically associative} (or \emph{B-associative} for short) if it satisfies the equation $F(\bfx\bfy\bfz)=F(\bfx F(\bfy)^{|\bfy|}\bfz)$ for any $\bfx,\bfy,\bfz\in X^*$.
\end{definition}

It is easy to see that any B-associative string function $F\colon X^*\to X^*$ is B-preassociative and hence can be factorized as indicated in Theorem~\ref{thm:fa7sfds}. Moreover, any B-associative string function satisfying $\ran(F_n)\subseteq X$ for every $n\geqslant 1$ is also such that $F_n$ is $1$-quasi-range-idempotent for every $n\geqslant 1$ (see \cite{MarTeh3}) and therefore it can be factorized as described in Example~\ref{ex:5dfdas}. In this case we have $\delta_{F_n}\circ F_n=F_n$, which shows that $\id|_{\ran(F_n)}\in Q(\delta_{F_n})$ for every $n\geqslant 1$. Therefore, from Example~\ref{ex:5dfdas} we immediately derive the following corollary.

\begin{corollary}
Let $F\colon X^*\to X^*$ be a B-associative function satisfying $\ran(F_n)\subseteq X$ for every $n\geqslant 1$. Then, for every $n\geqslant 1$, we have $F_n=f_n\circ H_n$, where $H\colon X^*\to X^*$ is the length-preserving associative function defined by $H_n(\bfx)=F_n(\bfx)^n$ for every $n\geqslant 1$ and $f_n\colon\ran(H_n)\to X$ is the one-to-one function defined by $f_n(x^n)=x$ for every $n\geqslant 1$.
\end{corollary}


We end this section by an additional application of Theorem~\ref{thm:fa7sfds}.

\begin{definition}
We say that a function $F\colon X^*\to Y$ has a \emph{componentwise defined kernel} if there exists a family $\{E_n:n\geqslant 1\}$ of equivalence relations on $X$ such that for any $n\geqslant 1$ and any $\bfx,\bfy\in X^n$, we have $F(\bfx)=F(\bfy)$ if and only if $(x_i,y_i)\in E_i$ for $i=1,\ldots,n$. In this case, we say that the family $\{E_n:n\geqslant 1\}$ \emph{defines the kernel of $F$ componentwise}.
\end{definition}

This concept can be interpreted, e.g., in decision making, as follows. A function $F\colon X^*\to Y$ has a componentwise defined kernel if the equivalence between two $n$-profiles $\bfx,\bfy\in X^n$ can be defined attributewise.

The following proposition and corollary give characterizations of those B-preassociative functions which have a componentwise defined kernel.

\begin{proposition}\label{prop:fa6s}
Assume AC and let $F\colon X^*\to Y$ have a kernel defined componentwise by the family $\{E_n:n\geqslant 1\}$ of equivalence relations on $X$. Then $F$ is B-preassociative if and only if $E_n\subseteq E_{n+1}$ for every $n\geqslant 1$.
\end{proposition}


\begin{proof}
Let $F\colon X^*\to Y$ be defined as indicated in the statement. For the necessity, suppose that $F$ is B-preassociative and let $(x,y)\in E_n$ for some $n\geqslant 1$. Then we have $F(x^n)=F(x^{n-1}y)$ and hence $F(x^{n+1})=F(x^ny)$ by B-preassociativity. It follows that $(x,y)\in E_{n+1}$. For the sufficiency, for any $n\geqslant 1$ and any $\bfx, \bfy \in X^n$ such that $F(\bfx)=F(\bfy)$, we have $F(\bfx\bfz)=F(\bfy\bfz)$ for every $\bfz \in X^*$ by definition of $F$. Since $E_n\subseteq E_{n+1}$ for every $n\geqslant 1$, we also have $F(\bfz\bfx)=F(\bfz\bfy)$ for every $\bfz \in X^*$. Therefore $F$ is B-preassociative.
\end{proof}

\begin{corollary}
Assume AC and let $F\colon X^* \to Y$ be a function. The following assertions are equivalent.
\begin{enumerate}
\item[(i)] F is B-preassociative and has a componentwise defined kernel.

\item[(ii)] There exists a sequence $(h_n)_{n\geqslant 1}$ of unary operations on $X$ and a sequence $(f_n)_{n\geqslant 1}$ of one-to-one maps $f_n\colon\{h_1(x_1)\cdots h_n(x_n): x_1\cdots x_n \in X^n\} \to Y$ such that $h_n\circ h_n=h_n$, $h_{n+1}\circ h_n=h_{n+1}$, and $F_n(\bfx)=f_n(h_1(x_1)\cdots h_n(x_n))$ for every $n\geqslant 1$ and every $\bfx \in X^n$.
\end{enumerate}
\end{corollary}
\begin{proof}

(i) $\Rightarrow$ (ii). By Proposition~\ref{prop:fa6s}, the kernel of $F$ is defined by some family of equivalence relations $\{E_n : n\geqslant 1\}$ on $X$ satisfying $E_n \subseteq E_{n+1}$ for every $n\geqslant 1$. For every $c\in X/E_n$, let $s_n(c)\in c$ be a representative of $c$ and define the map $h_n\colon X\to X$ by $h_n(x)=s_n(x/E_n)$. The map $g_n\colon\ran(F_n)\to X^n$ defined by $g_n(F(\bfx))=h_1(x_1)\cdots h_n(x_n)$ is a quasi-inverse of $F_n$. Indeed, since $(x_i,h_i(x_i))\in E_i$ for every $\bfx\in X^n$ and every $i\in\{1,\ldots,n\}$, we have
$$
(F_n\circ g_n\circ F_n)(x_1\cdots x_n) ~=~ F_n(h_1(x_1)\cdots h_n(x_n)) ~=~ F_n(x_1\cdots x_n).
$$

By Theorem~\ref{thm:fa7sfds}, setting $H_n=g_n \circ F_n$ for every $n\geqslant 1$, there is a one-to-one function $f_n\colon\ran(H_n) \to Y$ such that $F_n=f_n \circ H_n$ and such that the map $H\colon X^* \to X^*$ obtained by setting $H_0(\varepsilon)=\varepsilon$ is associative and length-preserving. The conclusion follows from Example~\ref{ex:68vffd}.

(ii) $\Rightarrow$ (i) By Example \ref{ex:68vffd} and Proposition \ref{prop:s7f65} we obtain that $F$ is B-preassociative. Moreover, the kernel of $F$ is defined by the family $\{\ker(h_i) : i \geqslant 1\}$ of equivalence relations on $X$.
\end{proof}

\section*{Acknowledgments}

This research is supported by the internal research project F1R-MTH-PUL-15MRO3 of the University of Luxembourg.



\begin{thebibliography}{99}

\bibitem{deF31}
B.~de Finetti.
\newblock Sul concetto di media.
\newblock {\em Giornale dell' Instituto Italiano degli Attari} 2(3):369--396, 1931.

\bibitem{FodRou94}
J.~Fodor and M.~Roubens.
\newblock {\em Fuzzy preference modelling and multicriteria decision support}.
\newblock Kluwer, Dordrecht, 1994.

\bibitem{GraMarMesPap09}
M.~Grabisch, J.-L.~Marichal, R.~Mesiar, and E.~Pap.
\newblock {\em Aggregation functions}.
\newblock Encyclopedia of Mathematics and its Applications, vol. 127.
\newblock Cambridge University Press, Cambridge, 2009.

\bibitem{Kol30}
A.~N.~Kolmogoroff.
\newblock Sur la notion de la moyenne. ({F}rench).
\newblock {\em Atti Accad. Naz. Lincei}, 12(6):388--391, 1930.

\bibitem{LehMarTeh}
E.~Lehtonen, J.-L.~Marichal, B.~Teheux.
\newblock Associative string functions.
\newblock {\em Asian-European Journal of Mathematics} 7(4):1450059 (18 pages), 2014.

\bibitem{MarTeh}
J.-L.~Marichal and B.~Teheux.
\newblock Associative and preassociative functions.
\newblock {\em Semigroup Forum} 89(2):431--442, 2014. (Improved version available at arxiv.org/abs/1309.7303v3).

\bibitem{MarTeh2}
J.-L.~Marichal and B.~Teheux.
\newblock Preassociative aggregation functions.
\newblock {\em Fuzzy Sets and Systems} 268:15--26, 2015.

\bibitem{MarTeh3}
J.-L.~Marichal and B.~Teheux.
\newblock Barycentrically associative and preassociative functions.
\newblock {\em Acta Mathematica Hungarica} 145(2):468--488, 2015.

\bibitem{Nag30}
M.~Nagumo.
\newblock {\"U}ber eine Klasse der Mittelwerte. ({G}erman).
\newblock {\em Japanese Journ. of Math.}, 7:71--79, 1930.

\bibitem{SchSkl83}
B.~Schweizer and A.~Sklar.
\newblock {\em Probabilistic metric spaces}.
\newblock North-Holland Series in Probability and Applied Mathematics.
\newblock North-Holland Publishing Co., New York, 1983.
\newblock (New edition in: Dover Publications, New York, 2005).

\end{thebibliography}
\end{document}